\documentclass[12pt]{amsart}
\usepackage[sans]{dsfont}
\usepackage{amsfonts,amssymb,amsbsy,amsmath,amsthm,enumerate}
\numberwithin{equation}{section}

\newtheorem{theorem}{Theorem}[section]

\newtheorem{lemma}[theorem]{Lemma}

\newtheorem{follow}[theorem]{Corollary}

\newtheorem{pr}[theorem]{Proposition}
\theoremstyle{definition}

\newcommand{\bel}{\begin{equation} \label}
\newcommand{\ee}{\end{equation}}
\newcommand{\one}{\mathds{1}}

\newcommand{\rd}{{\mathbb R}^{2}}
\newcommand{\re}{{\mathbb R}}

\newcommand{\N}{{\mathbb N}}

\newcommand{\A}{{\mathbb A}}
\newcommand{\Z}{{\mathbb Z}}

\def\beq{\begin{equation}}
\def\eeq{\end{equation}}
\newcommand{\bea}{\begin{eqnarray}}
\newcommand{\eea}{\end{eqnarray}}
\newcommand{\beas}{\begin{eqnarray*}}
\newcommand{\eeas}{\end{eqnarray*}}

\begin{document}
\title[Metric perturbations of the Landau Hamiltonian]{Local Spectral Asymptotics for Metric Perturbations of the Landau Hamiltonian}

\author[T.~Lungenstrass]{Tom\'as Lungenstrass}
\author[G.~Raikov]{Georgi Raikov}

\begin{abstract}
We consider metric perturbations of the Landau Hamiltonian. We investigate the asymptotic behaviour of the discrete spectrum of the perturbed operator near the Landau levels, for perturbations of compact support, and of exponential or power-like decay at infinity.
\end{abstract}

\maketitle



%%%%%%%%%%%%%%%%%%%%%%%%%%%%%%%%%%%%%%%%%%%%%%%%%%%%%%%%%%%%%%%%%%%%%%%%%%
%%%%%%%%%%%%%%%%%%%%%%%%%%%%%%%%%%%%%%%%%%%%%%%%%%%%%%%%%%%%%%%%%%%%%%%%%%

\section{Introduction}
\label{s1} \setcounter{equation}{0}
Let $$H_0 : = (-i\nabla - A_0)^2,$$
with $A_0 = (A_{0,1},A_{0,2}) : = \frac{b}{2} \left(-x_2, x_1\right)$,
    be the Landau Hamiltonian, self-adjoint in $L^2(\rd)$, and essentially self-adjoint on $C_0^\infty(\rd)$. In other words, $H_0$ is the 2D Schr\"odinger operator with constant scalar magnetic field $b>0$, i.e. the Hamiltonian of a 2D spinless non relativistic quantum particle subject to a constant magnetic field.
    As is well known, the spectrum $\sigma(H_0)$ consists of infinitely degenerate eigenvalues $\Lambda_q : = b(2q+1)$, $q \in \Z_+ : = \left\{0,1,2,\ldots\right\}$, called {\em Landau levels} (see e.g. \cite{f,l}).\\
    In the present article we consider metric perturbations of $H_0$. Namely, let
    $$
    m(x) = \left\{m_{jk}(x)\right\}_{j,k=1,2}, \quad x \in \rd,
    $$
    be a Hermitian $2 \times 2$ matrix such that $m(x) \geq 0$ for all $x \in \rd$. Throughout the article we assume that $m_{jk} \in C^\infty_{\rm b}(\rd)$, $j,k=1,2$, i.e. $m_{jk} \in C^\infty(\rd)$, and $m_{jk}$ together with all its  derivatives are bounded on $\rd$.
    Set
    \bel{0}
    \Pi_j : = -i\frac{\partial}{\partial x_j} - A_{0,j}, \quad j=1,2,
    \ee
     so that $H_0 =  \Pi_1^2 + \Pi_2^2$. On ${\rm Dom}\,H_0$ define the  operators
    $$
    H_\pm : = \sum_{j,k=1,2} \Pi_j (\delta_{jk} \pm m_{jk})\Pi_k = H_0 \pm W
    $$
    where $W : = \sum_{j,k=1,2} \Pi_j  m_{jk} \Pi_k$; in the case of $H_-$, we suppose additionally that $\sup_{x \in \rd}|m(x)| <1$. Thus the matrices $g_{\pm}(x) = \left\{g_{jk}^{\pm}(x)\right\}_{j,k=1,2}$ with $g_{jk}^{\pm} : =  \delta_{jk} \pm m_{jk}$ are positive definite for each $x \in \rd$. Under these assumptions, the operators $H_\pm$ are self-adjoint in $L^2(\rd)$, and essentially self-adjoint on $C^\infty_0(\rd)$ (see the Appendix).\\
    From mathematical physics point of view, the operators $H_\pm$ are special cases of Schr\"odinger operators with {\em position-dependent mass} which have been investigated since long ago (see e.g. \cite{bfm, roos}), but the interest towards which increased essentially during the last decade (see e.g. \cite{mrr, gs, k}). Here we would like to mention especially the article \cite{sddo} where the model considered is quite close to the operators $H_\pm$ discussed in the present paper. \\
    The operators $H_\pm$ admit also a geometric interpretation since they are related to the Bochner Laplacians corresponding to connections with constant non-vanishing curvature
    (see e.g. \cite{ros, cdv}); we discuss this relation in more detail at the end of Section \ref{s2}.
     Further, assume that
    \bel{d1}
    \lim_{|x| \to \infty} m_{jk}(x) = 0, \quad j,k=1,2.
    \ee
    Thus $m$ models a localized perturbation with respect to a reference medium.
    Under condition \eqref{d1} the resolvent difference $H_\pm^{-1} - H_0^{-1}$ is a compact operator (see the Appendix), and therefore the essential spectra of $H_\pm$ and $H_0$ coincide, i.e.
    $$
    \sigma_{\rm ess}(H_\pm) =  \sigma_{\rm ess}(H_0) = \sigma(H_0) = \bigcup_{q=0}^\infty\left\{\Lambda_q\right\}.
    $$
    The spectrum $\sigma(H_\pm)$ on $\re\setminus \bigcup_{q=0}^\infty\left\{\Lambda_q\right\}$ may consist of discrete eigenvalues whose only possible accumulation points are the Landau levels. Moreover, taking into account that $W \geq 0$, and applying \cite[Theorem 7, Section 9.4]{birsol}, we find that the eigenvalues of $H_+$ (resp., $H_-$) may accumulate to a given Landau level $\Lambda_q$ only from above (resp., from below). Fix $q \in \Z_+$. Let $\left\{\lambda_{k,q}^-\right\}$ be the eigenvalues of $H_-$ lying on the interval $(\Lambda_{q-1}, \Lambda_q)$ with $\Lambda_{-1} : = -\infty$, counted with the multiplicities, and enumerated in increasing order. Similarly, let $\left\{\lambda_{k,q}^+\right\}$ be the eigenvalues of $H_+$  lying on the interval $(\Lambda_q, \Lambda_{q+1})$, counted with the  multiplicities, and enumerated in decreasing order.\\
    The aim of the article is to investigate the rate of convergence of $\lambda_{k,q}^\pm - \Lambda_q $ as $k \to \infty$, $q \in \Z_+$ being fixed, for perturbations $m$ of compact support, of exponential decay, or of power-like decay at infinity.\\
    The properties of the discrete spectrum generated by perturbative second-order differential operators with decaying coefficients have been considered also in \cite{aadh, boylev, bkrs, r1}. \\
    The article is organized as follows. In  Section \ref{s2} we formulate our main results, and briefly comment on them. In Section  \ref{ss32} we reduce our analysis to the study of operators of Berezin--Toeplitz type, and in Section \ref{ss33} we establish several useful unitary equivalences for these operators. Section \ref{ss34} contains the proofs of our results in the case of rapid decay, i.e. of compact support or exponential decay, while  the proofs for slow, i.e. power-like decay, could be found in Section \ref{ss35}. Finally, in the Appendix we address some standard issues concerning the domain  of the operators $H_\pm$, and the compactness of the resolvent difference $H_0^{-1} - H_\pm^{-1}$.

\section{Main Results}
\label{s2} \setcounter{equation}{0}
First, we formulate our results concerning perturbations $m$ of compact support. Denote by  $m_<(x)$ and  $m_>(x)$ with $m_<(x) \leq m_>(x)$, the two eigenvalues of  the matrix $m(x)$, $x \in \rd$.

    \begin{theorem} \label{th1}
    Assume that the support of the matrix $m$ is compact, and its smaller eigenvalue $m_<$ does not vanish identically. Fix $q \in \Z_+$. Then we have
    \bel{1}
    \ln{\left(\pm\left(\lambda_{k,q}^\pm - \Lambda_q\right)\right)} = -k\ln{k} + O(k), \quad k \to \infty.
    \ee
    \end{theorem}
    {\em Remarks}: (i) Under additional technical hypotheses on $m_{\gtrless}$, we could make asymptotic relation \eqref{1} more precise.  Namely, assume
    that there exists a non increasing sequence $\left\{s_j\right\}_{j \in {\mathbb N}}$, such that $s_j >  0$, $j \in {\mathbb N}$, $\lim_{j \to \infty} s_j = 0$, and the level lines
    $$
    \left\{x \in \rd \, | \, m_<(x) = s_j\right\}, \quad j \in {\mathbb N},
    $$
    are bounded Lipschitz curves. In particular, the existence of such sequence follows from the Sard lemma (see e.g. \cite[Theorem 3.1, Chapter 2]{shst}) if we assume that $m_< \in C^2(\rd)$.
    Further, denote by ${\mathcal C}_{\gtrless}$ the {\em logarithmic capacities} (see e.g. \cite[Section 4, Chapter II]{land}) of ${\rm supp}\,m_\gtrless$. Then we have
\begin{multline}\label{2}
    \left(1 + \ln{\left(b{\mathcal C}_<^2/2\right)}\right)k + o(k) \leq\\
    \ln{\left(\pm\left(\lambda_{k,q}^\pm - \Lambda_q\right)\right)} + k\ln{k} \leq
    \left(1 + \ln{\left(b{\mathcal C}_>^2/2\right)}\right)k + o(k)
\end{multline}
    as $k \to \infty$. We omit the details of the proof of \eqref{2}, inspired by \cite{fp}.\\
    (ii) For $q \in \Z_+$ and $\lambda > 0$, set
    \bel{8}
    {\mathcal N}_q^{\pm}(\lambda) : = \#\left\{k \in \Z_+ \, | \, \pm\left(\lambda_{k,q}^\pm - \Lambda_q\right) > \lambda\right\}.
    \ee
    Then a less precise version of \eqref{1}, namely
    $$
    \ln{\left(\pm\left(\lambda_{k,q}^\pm - \Lambda_q\right)\right)} = -k\ln{k}\,(1 + o(1)), \quad k \to \infty,
    $$
    is equivalent to
    \bel{7}
     {\mathcal N}_q^{\pm}(\lambda) = \frac{|\ln{\lambda}|}{\ln{|\ln{\lambda}|}}(1 + o(1)), \quad \lambda \downarrow 0.
     \ee

     Further, we state our results concerning perturbations of exponential decay. Assume that there exist constants $\beta>0$ and $\gamma > 0$ such that
     \bel{3}
     \ln{m_{\gtrless}(x)} = - \gamma |x|^{2\beta} + O(\ln{|x|}), \quad |x| \to \infty.
     \ee
     {\em Remark}: In \eqref{3}, we suppose that the values of $\gamma$ and $\beta$ are the same for $m_<$ and $m_>$. Of course, the remainder
     $O(\ln{|x|})$ could be different for $m_<$ and $m_>$. \\

     Given $\beta>0$ and $\gamma > 0$, set $\mu : = \gamma(2/b)^\beta$, $b>0$ being the constant magnetic field.

     \begin{theorem} \label{th2}
     Let $m_{\gtrless}$ satisfy \eqref{3}. Fix $q \in \Z_+$.\\
     {\rm (i)} If $\beta \in (0,1)$, then there exist constants $f_j = f_j(\beta, \mu)$, $j \in \N$, with $f_1 = \mu$, such that
     \bel{4}
     \ln{\left(\pm\left(\lambda_{k,q}^\pm - \Lambda_q\right)\right)} = - \sum_{1 \leq j < \frac{1}{1-\beta}} f_j k^{(\beta-1)j + 1} + O(\ln{k}), \quad k \to \infty.
     \ee
     {\rm (ii)} If $\beta = 1$, then
     \bel{5}
     \ln{\left(\pm\left(\lambda_{k,q}^\pm - \Lambda_q\right)\right)} = - \left(\ln{(1+\mu)}\right) k + O(\ln{k}), \quad k \to \infty.
     \ee
     {\rm (iii)} If $\beta \in (1,\infty)$, then there exist constants $g_j = g_j(\beta, \mu)$, $j \in \N$,  such that
\begin{multline}\label{6}
     \ln{\left(\pm\left(\lambda_{k,q}^\pm - \Lambda_q\right)\right)} =
    - \frac{\beta - 1}{\beta} k \ln{k} \\+ \left(\frac{\beta - 1 - \ln{(\mu\beta)}}{\beta}\right) k  - \sum_{1 \leq j < \frac{\beta}{\beta-1}} g_j k^{(\frac{1}{\beta}-1)j + 1} + O(\ln{k}), \quad k \to \infty.
\end{multline}
      \end{theorem}
      {\em Remarks}: (i) Let us describe explicitly  the coefficients $f_j$ and $g_j$, $j \in \N$, appearing in \eqref{4} and \eqref{6} respectively.
      Assume at first $\beta \in (0,1)$. For $s>0$ and $\epsilon \in \re$, $|\epsilon| << 1$, introduce the function
      \bel{o1}
      F(s;\epsilon) : = s-\ln{s} + \epsilon \mu  s^\beta.
      \ee
      Denote by $s_<(\epsilon)$ the unique positive solution of the equation $s = 1 - \epsilon \beta \mu s^\beta$, so that $\frac{\partial F}{\partial s}(s_<(\epsilon); \epsilon) = 0$. Set
      \bel{o4}
      f(\epsilon) : = F(s_<(\epsilon); \epsilon).
      \ee
      Note that $f$ is a real analytic function for small $|\epsilon|$. Then $f_j: = \frac{1}{j!} \frac{d^j f}{d\epsilon^j}(0)$, $j \in \N$. \\
      Let now $\beta \in (1,\infty)$. For $s>0$ and $\epsilon \in \re$, $|\epsilon| << 1$, introduce the function
      \bel{o9}
      G(s;\epsilon) : = \mu s^\beta -\ln{s} + \epsilon s.
      \ee
      Denote by $s_>(\epsilon)$ the unique positive solution of the equation $\beta \mu s^\beta = 1 - \epsilon s$ so that $\frac{\partial G}{\partial s}(s_>(\epsilon); \epsilon) = 0$. Define
      \bel{o10}
      g(\epsilon) : = G(s_>(\epsilon); \epsilon),
      \ee
      which is a real analytic function for small $|\epsilon|$. Then $g_j: = \frac{1}{j!} \frac{d^j g}{d\epsilon^j}(0)$, $j \in \N$.\\
      (ii) If, instead of \eqref{3}, we assume that
      \bel{j1a}
     \ln{m_{\gtrless}(x)} = - \gamma |x|^{2\beta}(1 + o(1)), \quad |x| \to \infty,
     \ee
     then we can prove less precise versions of \eqref{4}, \eqref{5}, and \eqref{6}, namely
     $$
      \ln{\left(\pm\left(\lambda_{k,q}^\pm - \Lambda_q\right)\right)} =
\begin{cases}
      -\mu k^{\beta} (1 + o(1)) \quad {\rm if} \quad \beta \in (0,1),\\
      - \left(\ln{(1+\mu)}\right) \, k (1 + o(1)) \quad {\rm if} \quad \beta = 1,\\
      - \frac{\beta - 1}{\beta} k \ln{k} \, (1 + o(1))\quad {\rm if} \quad \beta \in (1,\infty),
\end{cases}
      \ k \to \infty,
      $$
      which are equivalent to
      \bel{o61}
      {\mathcal N}_q^\pm(\lambda) =
\begin{cases}
      \mu^{-1/\beta} |\ln{\lambda}|^{1/\beta} (1 + o(1)) \quad {\rm if} \quad \beta \in (0,1),\vspace{3pt}\\
      \frac{1}{\ln{(1+\mu)}} |\ln{\lambda}| (1 + o(1)) \quad {\rm if} \quad \beta = 1,\vspace{4pt}\\
      \frac{\beta}{\beta-1} \frac{|\ln{\lambda}|}{\ln{|\ln{\lambda}|}}(1 + o(1))\quad {\rm if} \quad \beta \in (1,\infty),
\end{cases}
      \ \lambda \downarrow 0.
      \ee
      Note that in \eqref{j1a}, similarly to \eqref{3}, we assume that the values of $\gamma$ and $\beta$ are the same for $m_<$ and $m_>$. However, since the coefficient in \eqref{o61} with $\beta > 1$ does not depend on $\gamma$, in this case we could assume different values of $\gamma > 0$ for
      $m_<$ and $m_>$. \\

      Finally, we consider perturbations $m$ which admit a power-like decay at infinity. For $\rho > 0$ recall the definition of the H\"ormander class
      $$
      {\mathcal S}^{-\rho}(\rd) : = \left\{\psi \in C^\infty(\rd)\, | \, |D^\alpha \psi(x)| \leq c_\alpha \langle x\rangle^{-\rho -|\alpha|}, \; x \in \rd, \; \alpha \in \Z_+^2\right\},
      $$
      where $\langle x\rangle : = (1 + |x|^2)^{1/2}$, $x \in \rd$. Let $\psi : \rd \to \re$ satisfy $\lim_{|x| \to \infty} \psi(x) = 0$. Set
      \bel{o60}
      \Phi_\psi(\lambda) : = \left|\left\{x \in \rd \, | \, \psi(x) > \lambda\right\}\right|, \quad \lambda > 0,
      \ee
      where $|\cdot |$ denotes the Lebesgue measure. Fix $q \in \Z_+$, and introduce the function
      \bel{o24}
      {\mathcal T}_q(x) : =
      \frac{1}{2} \left(\Lambda_q {\rm Tr}\,m(x) -  2b\, {\rm Im}\,m_{12}(x)\right),
      \quad x \in \rd.
      \ee
      Note that ${\mathcal T}_q(x) \geq 0$ for any $x \in \rd$ and $q \in \Z_+$.

      \begin{theorem} \label{th3}
      Let $m_{jk} \in {\mathcal S}^{-\rho}(\rd)$, $j,k =1,2$, with $\rho > 0$. Fix $q \in \Z_+$. Suppose that there exists a function $0 < \tau_q \in C^\infty({\mathbb S}^1)$, such that
      $$
      \lim_{|x| \to \infty} |x|^{\rho} {\mathcal T}_q(x) = \tau_q(x/|x|).
      $$
       Then we have
      \bel{10}
      {\mathcal N}_q^\pm(\lambda) = \frac{b}{2\pi} \Phi_{{\mathcal T}_q}(\lambda)(1 + o(1)) \asymp \lambda^{-2/\rho}, \quad \lambda \downarrow 0,
      \ee
      which is equivalent to
      \bel{11}
      \lim_{\lambda \downarrow 0} \lambda^{2/\rho} {\mathcal N}_q^\pm(\lambda) = {\mathcal C}_q : = \frac{b}{4\pi} \int_0^{2\pi} \tau_q(\cos{\theta}, \sin{\theta})^{2/\rho} d\theta,
      \ee
      or to
      \bel{j1}
       \pm\left(\lambda_{k,q}^\pm - \Lambda_q\right) = {\mathcal C}_q^{\rho/2} k^{-\rho/2}(1 + o(1)), \quad k \to \infty.
       \ee
      \end{theorem}
    {\em Remarks}: (i) Relation \eqref{10} could be regarded as a semiclassical one, although here the semiclassical interpretation is somewhat implicit. In Propositions \ref{pr2} and \ref{pr4} below, we show that the effective Hamiltonian which governs the asymptotics of ${\mathcal N}_q^\pm(\lambda)$ as $\lambda \downarrow 0$ is a pseudo-differential operator ($\Psi$DO) with anti-Wick symbol $w_{q, b} : = w_q \circ {\mathcal R}_b$, defined by \eqref{63} and \eqref{ggg}. Under the assumptions of Theorem \ref{th3}, ${\mathcal T}_{q,b} : =  {\mathcal T}_{q} \circ {\mathcal R}_b$ (see \eqref{o24} and \eqref{ggg}) can be considered as the principal part of the symbol $w_{q, b}$, while the difference between the anti-Wick and the Weyl quantization is negligible. Then
    $\frac{1}{2\pi} \Phi_{{\mathcal T}_{q,b}}(\lambda) = \frac{b}{2\pi} \Phi_{{\mathcal T}_q}(\lambda)$ is just the main semiclassical asymptotic term  for the eigenvalue counting function for a compact $\Psi$DO with Weyl symbol ${\mathcal T}_{q,b}$. \\
    (ii) There exists an extensive family of alternative sets of assumptions for Theorem \ref{th3} (see e.g. \cite{ivrii, daurob}). We have chosen here hypotheses which, for certain, are not the most general ones, but are quite explicit and, hopefully, easy to absorb.\\

Let us comment briefly on our results. Nowadays, there exists a relatively wide literature on the local spectral asymptotics for various magnetic quantum Hamiltonians. Let us concentrate here on three types of perturbations of $H_0$ which are considered to be of a particular  interest (see e.g. \cite{ivrii, mao}):
       \begin{itemize}
       \item
       Electric perturbations
       $H_0 + Q$
       where $Q : \rd \to \re $  plays the role of the perturbative {\em electric potential};
       \item
       Magnetic perturbations
       $(-i\nabla - A_0 - A)^2$
       where $A = (A_1, A_2)$, and $B : = \frac{\partial A_2}{\partial x_1} - \frac{\partial A_1}{\partial x_2}$  is  the  perturbative {\em magnetic field};
       \item
       Metric perturbations
       $ \sum_{j,k = 1,2} \Pi_j \left(\delta_{jk} + m_{jk} \right)\Pi_k$
       where $m = \left\{m_{jk}\right\}_{j,k=1,2}$ is  an appropriate perturbative matrix-valued function.
       \end{itemize}
       Typically, the perturbations $Q$, $B$, or $m$ are supposed to decay in a suitable sense at infinity.
       Slowly decaying $Q$, e.g. $Q \in  {\mathcal S}^{-\rho}(\rd)$ with $\rho > 0$ were considered in \cite{r0}, and the main asymptotic terms of the corresponding counting functions ${\mathcal N}_q^\pm(\lambda)$ as $\lambda \downarrow 0$ were found, utilizing, in particular, anti-Wick $\Psi$DOs  .
       In \cite[Theorem 11.3.17]{ivrii},  the case of {\em combined} electric, magnetic, and metric slowly decaying perturbations was investigated,  the main asymptotic terms of ${\mathcal N}_q^\pm(\lambda)$ as $\lambda \downarrow 0$, as well as certain remainder estimates were obtained. The semiclassical microlocal analysis applied in \cite{ivrii} imposed restrictions on the symbols involved which, in some sense or another, had to decay at infinity less rapidly than their derivatives. These restrictions did not allow  to handle some rapidly decaying perturbations, e.g. those of compact support, or of exponential decay with $\beta \geq 1/2$ (see \eqref{3}). \\
       In \cite{rw} the authors used a different approach based on the spectral analysis of Berezin--Toeplitz operators and obtained the main asymptotic terms of ${\mathcal N}_q^\pm(\lambda)$ as $\lambda \downarrow 0$ in the case of potential perturbations $Q$ of exponential decay or of compact support. In particular, in \cite{rw} formulas of type \eqref{7} or \eqref{o61} appeared for the first time. In the present article, we essentially improve the methods developed in \cite{rw}. These improvements lead also to more precise results for certain rapidly decaying electric perturbations. Namely, assume that $Q \geq 0$ admits a decay at infinity which is compatible in a suitable sense with the decay of $m$. Then the results of the article extend quite easily to operators of the form
       \bel{d2}
       H_\pm \pm Q,
       \ee
       so that $H_\pm \pm Q$ are perturbations of $H_0$ having a definite sign. We do not include these generalizations just in order to avoid an unreasonable increase of the size of the article due to  results which do not require any really new arguments. \\
       Combined perturbations of $H_0$ by compactly supported $B$ and $Q$ were considered in \cite{roztas} where  the main asymptotic terms of
       ${\mathcal N}_q^\pm(\lambda)$ as $\lambda \downarrow 0$ were found. Note that the magnetic perturbations of $H_0$ are never of fixed sign which creates specific difficulties, successfully overcome in \cite{roztas}. \\
       To authors' best knowledge, no results on the spectral asymptotics for rapidly decaying {\em metric} perturbations of $H_0$ appeared before in the literature. We also included in the article  our result on slowly-decaying metric perturbations (see Theorem \ref{th3}) since it is coherent with the unified approach of the article, and is proved by methods quite different from those in \cite{ivrii}.\\
       Finally, let us discuss briefly the relation of $H_\pm$ to the Bochner Laplacians. Assume that the elements of $m$ are real. In $\rd$ introduce a Riemannian metric generated by the inverse of $g^\pm$, and the connection 1-form $\sum_{j=1,2}A_{0,j} dx_j$. Set $\gamma_\pm : = \left({\rm det}\,g^\pm\right)^{-1/2}$. Then the standard Bochner Laplacian, self-adjoint in $L^2(\rd; \gamma_\pm dx)$, is written in local coordinates as
       $$
       {\mathcal L}_\pm : = - \gamma_\pm^{-1} \sum_{j,k = 1,2} \Pi_j g_{jk}^\pm  \gamma_\pm \Pi_k .
       $$
       Let $U_\pm : L^2(\rd; \gamma_\pm dx) \to L^2(\rd; dx)$ be the unitary operator defined by $U_\pm f : = \gamma_\pm^{-1/2}f$. Then we have
       \bel{d3}
       U_\pm {\mathcal L}_\pm  U_\pm^* = H_\pm + Q_\pm
       \ee
       where
       $$
       Q_\pm : = \frac{1}{4} \sum_{j,k=1,2}\left( g_{jk}^\pm  \frac{\partial \, \ln{\gamma_\pm}}{\partial x_k} \, \frac{\partial \, \ln{\gamma_\pm}}{\partial x_j}  - 2 \frac{\partial}{\partial x_j} \left(g_{jk}^\pm  \frac{\partial \, \ln{\gamma_\pm}}{\partial x_k} \right)\right).
       $$
       Generally speaking, the functions $Q_\pm $ do not have a definite sign coinciding with the sign of the operators $H_\pm - H_0$; hence, the operators on the r.h.s of \eqref{d3} are not exactly of the form of \eqref{d2}. The fact that the symbol of a Toeplitz operator does not have a definite sign may cause considerable difficulties in the study of the spectral asymptotics of this operator if the symbol decays rapidly and, in particular, when its support is compact (see e.g. \cite{pr}). Hopefully, we will overcome these difficulties in a future work where we would consider the local spectral asymptotics of ${\mathcal L}_\pm$.

%%%%%%%%%%%%%%%%%%%%%%%%%%%%%%%%%%%%%%%%%%%%%%%%%%%%%%%%%%%%%%%%%%%%%%%%%%
%%%%%%%%%%%%%%%%%%%%%%%%%%%%%%%%%%%%%%%%%%%%%%%%%%%%%%%%%%%%%%%%%%%%%%%%%%
\section{Reduction to Berezin-Toeplitz Operators}
\label{ss32}
\setcounter{equation}{0}
In this section we reduce the analysis of the functions ${\mathcal N}_q^\pm(\lambda)$ as $\lambda \downarrow 0$ to the spectral asymptotics for certain compact operators of Berezin-Toeplitz type. To this end, we will need some more notations, and several auxiliary results from the abstract theory of compact operators in Hilbert space.\\
In what follows, we denote by $\one_M$ the characteristic function of the set $M$.
Let $T$ be a self-adjoint operator in a Hilbert space\footnote{All the Hilbert spaces considered in the article are supposed to be separable.}, and ${\mathcal I} \subset \re$ be an interval. Set
$$
N_{\mathcal I}(T) : = {\rm rank}\,\one_{\mathcal I}(T),
$$
where, in accordance with our general notations, $\one_{\mathcal I}(T)$ is the spectral projection of $T$ corresponding to ${\mathcal I}$.
Thus, if ${\mathcal I}\cap \sigma_{\rm ess}(T) = \emptyset$, then $N_{\mathcal I}(T)$ is just the number of the eigenvalues of $T$, lying on ${\mathcal I}$, and counted with their multiplicities.
In particular,
     \bel{12}
    {\mathcal N}_q^-(\lambda) = N_{(\Lambda_{q-1}, \Lambda_q-\lambda)}(H_-), \quad q \in \Z_+, \quad \lambda \in  (0, 2b),
    \ee
    \bel{14}
    {\mathcal N}_q^+(\lambda) = N_{(\Lambda_{q} + \lambda, \Lambda_{q+1})}(H_+), \quad q \in \Z_+, \quad \lambda \in  (0, 2b),
    \ee
    the functions ${\mathcal N}_q^\pm$ being defined in \eqref{8}.
    Let $T=T^*$ be a linear compact operator in a Hilbert space. For $s>0$ set
    $$
    n_\pm(s; T) : = N_{(s,\infty)}(\pm T);
    $$
    thus, $n_+(s;T)$ (resp., $n_-(s;T)$) is just the number of the eigenvalues of the operator $T$ larger than $s$ (resp., smaller than $-s$), counted with their multiplicities. If $T_j = T_j^*$, $j=1,2$, are two linear compact operators, acting in a given Hilbert space, then the Weyl inequalities
    \bel{wi}
    n_{\pm}(s_1+s_2; T_1+T_2) \leq n_\pm(s_1; T_1) + n_\pm(s_2; T_2)
    \ee
    hold for $s_j > 0$ (see e.g. \cite[Theorem 9, Section 9.2]{birsol}).

%%%%%%%%%%%%%%%%%%%%%%%%%%%%%%%%%%%%%%%%%%%%%%%%%%%%%%%%%%%%%%%%%%%%%%%%%%%%%%
Fix $q \in \Z_+$ and denote by $P_q$ the orthogonal projection onto ${\rm Ker}\,(H_0-\Lambda_q)$. Since the operator $H_0^{-1}WH_0^{-1}$ is compact, the operator $P_q W P_q = \Lambda_q^2 P_q H_0^{-1}WH_0^{-1} P_q$ is compact as well. Similarly, the operators $H_0^{-1}WH_\pm^{-1/2}$ are compact, and hence the operators
$$
P_q W H_\pm^{-1} W P_q = \Lambda_q^2 P_q (H_0^{-1} W H_\pm^{-1/2})(H_\pm^{-1/2} W H_0^{-1}) P_q
$$
are compact as well.

\begin{pr} \label{pr1}
Under the general assumptions of the article we have
$$
n_+((1+\varepsilon)\lambda; P_qWP_q \mp P_q W H_\pm^{-1} W P_q) + O(1) \leq
$$
    \bel{19}
{\mathcal N}_q^\pm(\lambda) \leq
    \ee
$$
n_+((1-\varepsilon)\lambda; P_qWP_q \mp P_q W H_\pm^{-1} W P_q) + O(1), \quad \lambda \downarrow 0,
$$
    for each $\varepsilon \in (0,1)$.
\end{pr}
\begin{proof} The argument is close in spirit to the proof of \cite[Proposition 4.1]{rw}, and is based again on the (generalized) Birman--Schwinger principle. However, since the operator $H_0^{-1/2} W H_0^{-1/2}$ is only bounded but not compact, we cannot apply the Birman--Schwinger principle to the operator pair $(H_0,H_\pm)$, and apply it instead to the resolvent pair $(H_0^{-1}, H_{\pm}^{-1})$.
First of all, note that there exist $\Lambda_-$ and $\Lambda_+$ with $\Lambda_- \in (0, \Lambda_0)$ if $q = 0$, $\Lambda_- \in (\Lambda_{q-1}, \Lambda_q)$ if $q \in \N$, and $\Lambda_+ \in (\Lambda_q, \Lambda_{q+1})$ if $q \in \Z_+$, such that
    \bel{20}
    {\mathcal N}_q^-(\lambda) = N_{(\Lambda_-, \Lambda_q-\lambda)}(H_-), \quad \lambda \in (0, \Lambda_q - \Lambda_-),
    \ee
     \bel{21}
    {\mathcal N}_q^+(\lambda) = N_{(\Lambda_q+\lambda,\Lambda_+)}(H_+), \quad \lambda \in (0, \Lambda_+ - \Lambda_q).
    \ee
    Further, evidently,
    \bel{22}
    N_{(\Lambda_-, \Lambda_q-\lambda)}(H_-) = N_{((\Lambda_q-\lambda)^{-1},\Lambda_-^{-1})}(H_-^{-1}) = N_{((\Lambda_q-\lambda)^{-1},\Lambda_-^{-1})}(H_0^{-1}+T_-),
    \ee
     \bel{23}
    N_{(\Lambda_q+\lambda,\Lambda_+)}(H_+) = N_{(\Lambda_+^{-1},(\Lambda_q+\lambda)^{-1})}(H_+^{-1}) = N_{(\Lambda_+^{-1},(\Lambda_q+\lambda)^{-1})}(H_0^{-1}-T_+) ,
    \ee
    with $T_-: = H_-^{-1} - H_0^{-1}$ and $T_+: = H_0^{-1} - H_+^{-1}$. Note that the operators $T_\pm$ are non negative and compact. By the generalized Birman--Schwinger principle
    (see e.g. \cite[Theorem 1.3]{adh}) we have
    \begin{align}
    N_{((\Lambda_q-\lambda)^{-1},\Lambda_-^{-1})}(H_0^{-1}+T_-) &=
    n_+(1; T_-^{1/2}((\Lambda_q-\lambda)^{-1} - H_0^{-1})^{-1}T_-^{1/2})\nonumber\\ & - n_+(1; T_-^{1/2}(\Lambda_-^{-1} - H_0^{-1})^{-1}T_-^{1/2})\nonumber\\ &- {\rm dim \; Ker}\,(H_- - \Lambda_-),
    \label{24}
    \end{align}
    \begin{align}
    N_{(\Lambda_+^{-1},(\Lambda_q+\lambda)^{-1})}(H_0^{-1}-T_+)&=
    n_+(1; T_+^{1/2}(H_0^{-1} - (\Lambda_q+\lambda)^{-1})^{-1}T_+^{1/2})\nonumber\\ & - n_+(1; T_+^{1/2}(H_0^{-1}-\Lambda_+^{-1})^{-1}T_+^{1/2})\nonumber\\ &- {\rm dim \; Ker}\,(H_+ - \Lambda_+).
    \label{25}
    \end{align}
    Since the operators $T_\pm$ are compact, and $\Lambda_\pm \not \in \sigma(H_0)$, we find that the two last terms on the r.h.s. of \eqref{24} and \eqref{25} which are independent of $\lambda$, are finite. Next, the Weyl inequalities \eqref{wi} imply
\begin{multline}\label{26}
    n_+(1+\varepsilon; T_-^{1/2}((\Lambda_q-\lambda)^{-1} - H_0^{-1})^{-1}P_q T_-^{1/2}) \\- n_-(\varepsilon; T_-^{1/2}((\Lambda_q-\lambda)^{-1} - H_0^{-1})^{-1}(I-P_q)T_-^{1/2}) \leq\\
    n_+(1; T_-^{1/2}((\Lambda_q-\lambda)^{-1} - H_0^{-1})^{-1}T_-^{1/2}) \leq\\
    n_+(1-\varepsilon; T_-^{1/2}((\Lambda_q-\lambda)^{-1} - H_0^{-1})^{-1}P_q T_-^{1/2}) \\+ n_+(\varepsilon; T_-^{1/2}((\Lambda_q-\lambda)^{-1} - H_0^{-1})^{-1}(I-P_q)T_-^{1/2})
\end{multline}
    for any $\varepsilon \in (0,1)$. The operator $T_-^{1/2}((\Lambda_q-\lambda)^{-1} - H_0^{-1})^{-1}(I-P_q)T_-^{1/2}$ tends in norm as $\lambda \downarrow 0$ to the compact operator
    $$T_-^{1/2}\left(\sum_{j \in \Z_+\setminus\{q\}} (\Lambda_q^{-1} -\Lambda_j^{-1})^{-1} P_j \right)T_-^{1/2}.
    $$
    Therefore,
    \bel{28}
    n_\pm(\varepsilon; T_-^{1/2}((\Lambda_q-\lambda)^{-1} - H_0^{-1})^{-1}(I-P_q)T_-^{1/2}) = O(1), \quad \lambda \downarrow 0,
    \ee
    for any $\varepsilon > 0$. Next, for any $s>0$ we have
\begin{multline}\label{29}
    n_+(s; T_-^{1/2}((\Lambda_q-\lambda)^{-1} - H_0^{-1})^{-1}P_q T_-^{1/2}) =\\
    n_+(s; ((\Lambda_q-\lambda)^{-1} - \Lambda_q^{-1})^{-1}T_-^{1/2}P_q T_-^{1/2}) =
    n_+(s\lambda (\Lambda_q-\lambda)^{-1}\Lambda_q^{-1}; P_q T_- P_q).
\end{multline}
    Hence, \eqref{24} and \eqref{26} - \eqref{29} yield
    $$
    n_+((1+\varepsilon)\lambda (\Lambda_q-\lambda)^{-1}\Lambda_q^{-1}; P_q T_- P_q) + O(1) \leq
    $$
    $$
    N_{((\Lambda_q-\lambda)^{-1},\Lambda_-^{-1})}(H_0^{-1}+T_-) \leq
    $$
    \bel{30}
    n_+((1-\varepsilon)\lambda (\Lambda_q-\lambda)^{-1}\Lambda_q^{-1}; P_q T_- P_q) + O(1), \quad \lambda \downarrow 0,
    \ee
    for any $\varepsilon \in (0,1)$. Similarly, \eqref{25} and the analogues of \eqref{26} - \eqref{29} for positive perturbations, imply
    $$
    n_+((1+\varepsilon)\lambda (\Lambda_q+\lambda)^{-1}\Lambda_q^{-1}; P_q T_+ P_q) + O(1) \leq
    $$
    $$
    N_{(\Lambda_+^{-1},(\Lambda_q+\lambda)^{-1})}(H_0^{-1}-T_+) \leq
    $$
    \bel{31}
    n_+((1-\varepsilon)\lambda (\Lambda_q+\lambda)^{-1}\Lambda_q^{-1}; P_q T_+ P_q) + O(1), \quad \lambda \downarrow 0.
    \ee
    By the resolvent identity, we have
    $T_\pm = H_0^{-1} W H_0^{-1} \mp H_0^{-1} W H_\pm^{-1} W H_0^{-1}$,
    so that
    $$
    P_q T_\pm P_q = \Lambda_q^{-2}(P_q W P_q \mp P_q W H_\pm^{-1} W P_q).
    $$
    Thus,
     \bel{32}
    n_+(s;P_q T_\pm P_q) = n_+(s\Lambda_q^{2}; P_q W P_q \mp P_q W H_\pm^{-1} W P_q), \quad s>0.
    \ee
    Putting together \eqref{20} -- \eqref{23} and \eqref{30} -- \eqref{32}, we easily obtain \eqref{19}.
      \end{proof}
      \section{Unitary Equivalence for Berezin-Toeplitz Operators}
\label{ss33}
\setcounter{equation}{0}
%%%%%%%%%%%%%%%%%%%%%%%%%%%%%%%%%%%%%%%%%%%%%%%%%%%%%%%%%%%%%%%%%%%%%%%%%%%%%%
       Our first goal in this section is to show that under certain regularity conditions on the matrix $m$, the operator $P_q W P_q$, $q \in \Z_+$, with domain $P_q L^2(\rd)$, is unitarily equivalent to $P_0 w_q P_0$ with domain $P_0 L^2(\rd)$, where $w_q$ is the multiplier by a suitable function $w_q : \rd \to {\mathbb C}$. In fact, we will need a slightly more general result, and that is why  we introduce at first the appropriate notations. \\
      As usual, for $x = (x_1, x_2) \in \rd$ we set $z: = x_1 + ix_2$ and $\overline{z}: = x_1 - ix_2$ so that
      $$
      \frac{\partial}{\partial z} = \frac{1}{2}\left(\frac{\partial}{\partial x_1} -i \frac{\partial}{\partial x_2}\right), \quad  \frac{\partial}{\partial \overline{z}} = \frac{1}{2}\left(\frac{\partial}{\partial x_1} +i \frac{\partial}{\partial x_2}\right).
      $$
      Introduce the magnetic annihilation operator
      $$
      a: = -2i e^{-b|x|^2/4} \frac{\partial}{\partial \overline{z}} e^{b|x|^2/4} = - 2i \left(\frac{\partial}{\partial \overline{z}} + \frac{bz}{4}\right),
      $$
      and the magnetic creation operator
      $$
      a^*: = -2i e^{b|x|^2/4} \frac{\partial}{\partial z} e^{-b|x|^2/4} = - 2i \left(\frac{\partial}{\partial z} - \frac{b\overline{z}}{4}\right),
      $$
      with common domain ${\rm Dom}\,a = {\rm Dom}\,a^* = {\rm Dom}\,H_0^{1/2}$. The operators $a$ and $a^*$ are closed and mutually adjoint in $L^2(\rd)$. On
      ${\rm Dom}\,H_0$ we have $[a,a^*] = 2b$ and
      \bel{34}
      H_0 = a^* a + b = a a^* - b = \frac{1}{2} (a a^* + a^* a).
      \ee
      Moreover, on ${\rm Dom}\,H_0^{1/2}$ we have
      \bel{35}
      \Pi_1 = \frac{1}{2}(a+a^*), \quad \Pi_2 = \frac{1}{2i}(a-a^*),
      \ee
      the operators $\Pi_j$, $j=1,2$, being introduced in \eqref{0}. Next, define the operator $\A : {\rm Dom}\,H_0^{1/2} \to L^2(\rd; {\mathbb C}^2)$ by
      $$
      \A u : = \left(
      \begin{array} {c}
      a^*u \\
      au
      \end{array}
      \right), \quad u \in  {\rm Dom}\,H_0^{1/2}.
      $$
      Then, \eqref{34} implies that $H_0 = \frac{1}{2} \A^* \A$. Further, introduce the Hermitian matrix-valued function
      $$
      \Omega : = \left(
      \begin{array} {cc}
      \omega_{11} & \omega_{12} \\
      \omega_{21} & \omega_{22}
      \end{array}
      \right),
      $$
      with $\omega_{jk} \in L^{\infty}(\rd)$, $j,k=1,2$. Fix $q \in \Z_+$ and define the operator
      \bel{59}
      P_q  \A^* \Omega \A P_q = \Lambda_q P_q H_0^{-1/2} \A^* \Omega \A H_0^{-1/2} P_q,
      \ee
      bounded and self-adjoint in $P_q L^2(\rd)$. Utilizing \eqref{35}, we easily find that
      \bel{57}
      P_q W P_q = \frac{1}{2} P_q  \A^* U \A P_q
      \ee
      where
      \bel{58}
      U : = {\mathcal O}^* m {\mathcal O}, \quad {\mathcal O} : = \frac{1}{\sqrt{2}} \left(
      \begin{array} {cc}
      1 & 1 \\
      i & -i
      \end{array}
      \right),
      \ee
      so that $U = \left(
      \begin{array} {cc}
      u_{11} & u_{12} \\
      u_{21} & u_{22}
      \end{array}
      \right),
      $
      with
      $$
      u_{11} : = \frac{1}{2} \left({\rm Tr}\,m - 2 {\rm Im}\,m_{12}\right), \quad u_{22} : = \frac{1}{2} \left({\rm Tr}\,m + 2 {\rm Im}\,m_{12}\right),
      $$
      $$
      u_{12} = \overline{u_{21}} : = \frac{1}{2} \left(m_{11} - m_{22} - 2i {\rm Re}\,m_{12}\right).
      $$
      Introduce the Laguerre polynomials
      \bel{39}
      {\rm L}_q^{(m)} : = \sum_{j=0}^q \binom{q+m}{q-j} \frac{(-t)^j}{j!}, \quad t \in \re, \quad q \in \Z_+, \quad m \in \Z_+;
      \ee
      as usual, we write ${\rm L}_q^{(0)} = {\rm L}_q$, and for notational convenience we set $q {\rm L}_{q-1} = 0$ for $q=0$. By \cite[Eq. 8.974.3]{gr} we have
      \bel{48}
      \sum_{j=0}^q {\rm L}_j^{(m)}(t) = {\rm L}_q^{(m+1)}(t), \quad t \in \re, \quad q \in \Z_+, \quad m \in \Z_+.
       \ee
       \begin{pr} \label{pr2}
       Let $\Omega$ be a Hermitian $2 \times 2$ matrix-valued function with entries $\omega_{jk} \in C^{\infty}_{\rm b}(\rd)$, $j,k=1,2$.
       Fix $q \in \Z_+$. Then the operator $P_q \A^* \Omega \A P_q$ with domain $P_q L^2(\rd)$, is unitarily equivalent to the operator $P_0 w_q P_0$ with domain $P_0 L^2(\rd)$ where
       \bel{63}
       w_q = w_q(\Omega) :=
\begin{cases}
       2b(q+1) {\rm L}_{q+1}\left(-\frac{\Delta}{2b}\right) \omega_{11} +
       2bq {\rm L}_{q-1}\left(-\frac{\Delta}{2b}\right) \omega_{22} \vspace{3pt}\\\mkern150mu- 8 {\rm Re}\,{\rm L}_{q-1}^{(2)}\left(-\frac{\Delta}{2b}\right)\frac{\partial^2\omega_{12}}{\partial \overline{z}^2} \quad {\rm if} \quad q \geq 1,\\
       2b {\rm L}_{1}\left(-\frac{\Delta}{2b}\right) \omega_{11} \quad {\rm if} \quad q = 0,
\end{cases}
       \ee
       $\Delta$ is the standard Laplacian in $\rd$ so that, in accordance with \eqref{39}, ${\rm L}_{s}^{(m)}\left(-\frac{\Delta}{2b}\right)$ with $s \in \Z_+$ and $m \in \Z_+$, is just the differential operation $\sum_{j=0}^s \binom{s+m}{s-j} \frac{\Delta^j}{j!(2b)^j}$ of order $2s$ with constant coefficients.
       \end{pr}
       \begin{proof}
       Set
       $$
       \varphi_{0,k}(x) : = \sqrt{\frac{b}{2\pi k!}} \left(\frac{b}{2}\right)^{k/2}z^k e^{-b|x|^2/4}, \quad x \in \rd, \quad k \in \Z_+,
       $$
       $$
       \varphi_{q,k}(x) : = \sqrt{\frac{1}{(2b)^q q!}}(a^*)^q\varphi_{0,k}(x), \quad x \in \rd, \quad k \in \Z_+, \quad q \in {\mathbb N}.
       $$
       Then $\left\{\varphi_{q,k}\right\}_{k \in \Z_+}$ is an orthonormal basis of $P_q L^2(\rd)$ called sometimes {\em the angular momentum basis}
       (see e.g. \cite{rw} or \cite[Subsection 9.1]{bpr}). Evidently, for $k \in \Z_+$ we have
       \bel{37}
       a^* \varphi_{q,k} = \sqrt{2b(q+1)}  \varphi_{q+1,k}, \quad q \in \Z_+,
       \quad a \varphi_{q,k} =
\begin{cases}
       \sqrt{2bq}  \varphi_{q-1,k}, \quad q \geq 1,\\
       0 , \quad q = 0.
\end{cases}
       \ee
       Define the unitary operator ${\mathcal W}: P_q L^2(\rd) \to P_0 L^2(\rd)$ by ${\mathcal W} : u \mapsto v$
       where
       \bel{38}
       u = \sum_{k  \in \Z_+} c_k \varphi_{q,k}, \quad v = \sum_{k  \in \Z_+} c_k \varphi_{0,k}, \quad \{c_k\}_{k \in \Z_+} \in \ell^2(\Z_+).
       \ee
       We will show that
       \bel{36}
       P_q \A^* \Omega \A P_q = {\mathcal W}^* P_0 w_q P_0 {\mathcal W}.
       \ee
       For $V \in C_{\rm b}^{\infty}(\rd)$, $m,s \in \Z_+$, and $k,\ell \in \Z_+$, set
       $$
       \Xi_{m,s}(V;k,\ell) : = \langle V \varphi_{m,k}, \varphi_{s,\ell}\rangle
       $$
       where $\langle\cdot,\cdot\rangle$ denotes the scalar product in $L^2(\rd)$. Taking into account \eqref{37} and \eqref{38}, we easily find that
       $$
       \langle P_q \A^* \Omega \A P_q u, u\rangle =
       $$
       $$
       2b \sum_{k \in \Z_+} \sum_{\ell \in \Z_+} \left((q+1)\Xi_{q+1,q+1}(\omega_{11};k,\ell) + q\Xi_{q-1,q-1}(\omega_{22};k,\ell)\right)c_k \overline{c_\ell} \;  +
       $$
       \bel{56}
       2b\sqrt{q(q+1)} 2{\rm Re} \, \sum_{k \in \Z_+} \sum_{\ell \in \Z_+} \Xi_{q+1,q-1}(\omega_{21};k,\ell)c_k \overline{c_\ell},
       \ee
       if $q \geq 1$, and
       \bel{41}
       \langle P_0 \A^* \Omega \A P_0 u, u\rangle = 2b \sum_{k \in \Z_+} \sum_{\ell \in \Z_+} \Xi_{1,1}(\omega_{11};k,\ell) c_k \overline{c_\ell}.
       \ee
       Moreover,
       \bel{42}
       \langle P_0 w_q P_0 v, v\rangle = \sum_{k \in \Z_+} \sum_{\ell \in \Z_+} \Xi_{0,0}(w_q;k,\ell) c_k \overline{c_\ell}, \quad q \in \Z_+.
       \ee
       In \cite[Lemma 9.2]{bpr} (see also the remark after Eq.(2.2) in \cite{bbr}), it was shown that
       \bel{40}
       \Xi_{m,m}(V;k,\ell) = \Xi_{0,0}\left({\rm L}_m\left(-\frac{\Delta}{2b}\right)V;k,\ell\right), \quad m \in \Z_+.
       \ee
       Now \eqref{41}, \eqref{40} with $m=1$ and $V = \omega_{11}$, and \eqref{42} with $q=0$, imply \eqref{36}  in the case $q=0$.
       Assume $q \geq 1$. By \eqref{40}, we have
      \bel{43}
        \Xi_{q+1,q+1}(\omega_{11};k,\ell) = \Xi_{0,0}\left({\rm L}_{q+1}\left(-\frac{\Delta}{2b}\right)\omega_{11};k,\ell\right),
         \ee
         \bel{44}
         \Xi_{q-1,q-1}(\omega_{22};k,\ell) = \Xi_{0,0}\left({\rm L}_{q-1}\left(-\frac{\Delta}{2b}\right)\omega_{22};k,\ell\right).
       \ee
       Let us now consider the quantity $\Xi_{q+1,q-1}(V;k,\ell)$. Using \eqref{37}, we easily find that for $q \geq 2$ we have
       \bel{45}
       \Xi_{q+1,q-1}(V;k,\ell) = \frac{1}{\sqrt{2b(q+1)}} \Xi_{q,q-1}([V,a^*];k,\ell) + \sqrt{\frac{q-1}{q+1}}  \Xi_{q,q-2}(V;k,\ell),
       \ee
\begin{multline}\label{46}
       \Xi_{q,q-1}([V,a^*];k,\ell) = \frac{1}{\sqrt{2bq}} \Xi_{q-1,q-1}([[V,a^*],a^*];k,\ell) \\+ \sqrt{\frac{q-1}{q}} \Xi_{q-1,q-2}([V,a^*];k,\ell).
\end{multline}
       Moreover, $[V,a^*] = 2i\frac{\partial V}{\partial z}$, and
       \bel{47}
       [[V,a^*],a^*] = -4\frac{\partial^2 V}{\partial z^2}.
       \ee
       Using \eqref{46}, it is not difficult to prove by induction that
       \bel{49}
       \Xi_{q,q-1}([V,a^*];k,\ell) = \frac{1}{\sqrt{2bq}}\sum_{j=0}^{q-1} \Xi_{j,j}([[V,a^*],a^*];k,\ell), \quad q\geq1.
       \ee
       Now \eqref{40}, \eqref{47}, and \eqref{48} imply
       $$
       \sum_{j=0}^{q-1} \Xi_{j,j}([[V,a^*],a^*];k,\ell) = \sum_{j=0}^{q-1} \Xi_{0,0}\left(-4{\rm L}_{j}\left(-\frac{\Delta}{2b}\right)\frac{\partial^2 V}{\partial z^2};k,\ell\right) =
       $$
       \bel{50}
       \Xi_{0,0}\left(-4{\rm L}_{q-1}^{(1)}\left(-\frac{\Delta}{2b}\right)\frac{\partial^2 V}{\partial z^2};k,\ell\right).
       \ee
       Setting
       \bel{51}
       {\mathcal D}_q : = -4{\rm L}_{q-1}^{(1)}\left(-\frac{\Delta}{2b}\right)\frac{\partial^2 }{\partial z^2}, \quad q \in {\mathbb N},
       \ee
       we find that \eqref{49} and \eqref{50} imply
       \bel{52}
       \Xi_{q,q-1}([V,a^*];k,\ell) =  \frac{1}{\sqrt{2bq}} \Xi_{0,0}\left({\mathcal D}_q V;k,\ell\right).
       \ee
       Bearing in mind \eqref{45}, \eqref{40}, and \eqref{52}, it is not difficult to prove by induction that
       \bel{53}
       \Xi_{q+1,q-1}(V;k,\ell) = \frac{1}{2b\sqrt{q(q+1)}}\sum_{s=1}^q \Xi_{0,0}\left({\mathcal D}_{s} V;k,\ell\right).
       \ee
       Note that \eqref{48} and \eqref{53} imply
       \bel{54}
       \sum_{s=1}^q {\mathcal D}_{s} = -4{\rm L}_{q-1}^{(2)}\left(-\frac{\Delta}{2b}\right)\frac{\partial^2 }{\partial z^2}.
       \ee
       Now, \eqref{53} and \eqref{54} entail
       \bel{55}
       2b\sqrt{q(q+1)} \Xi_{q+1,q-1}(\omega_{21};k,\ell) = \Xi_{0,0}\left(-4{\rm L}_{q-1}^{(2)}\left(-\frac{\Delta}{2b}\right)\frac{\partial^2 \omega_{21}}{\partial z^2}; k,\ell\right).
       \ee
       Finally, \eqref{56} and \eqref{42} combined with \eqref{43}, \eqref{44}, and \eqref{55}, yield \eqref{36} with $q \geq 1$.
       \end{proof}
       In the rest of the section we establish  two other suitable representations for  the operators $P_q V P_q$, $q \in \Z_+$, with $V : \rd \to {\mathbb C}$.
       \begin{pr} \label{pr3}
       {\rm (i) \cite[Lemma 3.1]{fr}, \cite[Subsection 2.3]{bbr}} Let $V \in L_{\rm loc}^1(\rd)$ satisfy $\lim_{|x| \to \infty}V(x) = 0$. Then for each $q \in \Z_+$ the operator $P_q V P_q$ is compact. \\
       {\rm (ii) \cite[Lemma 3.3]{rw}} Assume in addition that $V$ is radially symmetric, i.e. there exists $v :[0,\infty) \to {\mathbb C}$ such that $V(x) = v(|x|)$, $x \in \rd$. Then the eigenvalues of the operator $P_q V P_q$ with domain $P_q L^2(\rd)$, counted with the multiplicities, coincide with the set
        \bel{53b}
       \left\{\langle V\varphi_{q,k}, \varphi_{q,k}\rangle\right\}_{k \in \Z_+}.
       \ee
       In particular, the eigenvalues of $P_0VP_0$ coincide with
       \bel{53a}
       \frac{1}{k!} \int_0^\infty v((2t/b)^{1/2}) e^{-t} t^k dt, \quad k \in \Z_+.
       \ee
       \end{pr}
       {\em Remarks}: (i) Let us recall that if $f$ is, say, a bounded function of exponential decay, then
       $$
       ({\mathcal M}f)(z) : = \int_0^\infty f(t) t^{z-1} dt, \quad z \in {\mathbb C}, \quad {\rm Re}\,z > 0,
       $$
       is called sometimes {\em the Mellin transform} of $f$. Some of the asymptotic properties as $k \to \infty$ of the integrals \eqref{53a} which we will later obtain and use in the proofs of Theorem \ref{th1} and \ref{th2}, could possibly be deduced from the general theory of the Mellin transform.\\
       (ii) Combining Propositions \ref{pr2} and \ref{pr3}, we find that if the matrix-valued function $\Omega$ is radially symmetric and diagonal, then the operator $P_q \A^* \Omega \A P_q$ acting in $P_q L^2(\rd)$ is unitarily equivalent to a {\em diagonal} operator in $\ell^2(\Z_+)$. If $\Omega$ is just radially symmetric, then $P_q \A^* \Omega \A P_q$ is unitarily equivalent to a {\em tridiagonal} operator acting in $\ell^2(\Z_+)$.\\

       The last proposition in this section concerns the unitary equivalence between the Berezin-Toeplitz operator $P_0 W P_0$ and a certain Weyl pseudo-differential operator ($\Psi$DO). Let us recall the definition of Weyl $\Psi$DOs acting in $L^2(\re)$. Denote by $\Gamma(\rd)$ the set of functions $\psi: \rd \to  {\mathbb C}$ such that
       $$
       \|\psi\|_{\Gamma(\rd)} : = \sup_{(y,\eta) \in \rd} \sup_{\ell, m = 0,1} \left|\frac{\partial^{\ell + m}\psi(y,\eta)}{\partial y^\ell \partial \eta^m}\right| < \infty.
       $$
       Then the operator ${\rm Op}^{\rm w}(\psi)$ defined initially as a mapping between the Schwartz class ${\mathcal S}(\re)$ and its dual class ${\mathcal S}'(\re)$ by
       $$
       \left({\rm Op}^{\rm w}(\psi)u\right)(y) = \frac{1}{2\pi} \int_\re \int_\re \psi\left(\frac{y+y'}{2},\eta\right)e^{i(y-y')\eta} u(y') dy'd\eta, \quad y \in \re,
       $$
       extends uniquely to an operator bounded in $L^2(\re)$. Moreover, there exists a constant $c$ such that
       \bel{56a}
       \|{\rm Op}^{\rm w}(\psi)\| \leq c \|\psi\|_{\Gamma(\rd)}
       \ee
       (see e.g. \cite[Corollary 2.5(i)]{abd}). \\
       {\em Remark}: Inequalities  of type \eqref{56a} are known as {\em Calder\'on-Vaillancourt} estimates. \\

    Put
    \bel{ggg}
    {\mathcal R}_b : = -b^{-1/2} \left(
    \begin{array} {cc}
    0 & 1\\
    1 & 0
    \end{array}
    \right),
    \ee
       and for $V : \rd \to {\mathbb C}$, define
      $$
       V_b(x) : = V\left({\mathcal R}_b x\right), \quad x \in \rd, \quad b>0.
       $$
       Moreover, set
       ${\mathcal G}(x) : = \frac{e^{-|x|^2}}{\pi}$, $x \in \rd$.

       \begin{pr} \label{pr4} {\em \cite[Theorem 2.11, Corollary 2.8]{prvb}} Let $V \in L^1(\rd) + L^\infty(\rd)$. Then the operator $P_0 V P_0$ with domain $P_0 L^2(\rd)$ is unitarily equivalent to the operator ${\rm Op}^{\rm w}(V_b * {\mathcal G})$.
       \end{pr}
       {\em Remark}: The operator ${\rm Op}^{\rm aw}(\psi)$ : = ${\rm Op}^{\rm w}(\psi * {\mathcal G})$ is called $\Psi$DO with {\em anti-Wick symbol} $\psi$
       (see e.g. \cite[Section 24]{shu}).

\section{Proofs of Theorem \ref{th1} and Theorem \ref{th2}}
\label{ss34} \setcounter{equation}{0}
In this section we complete the proofs of Theorem \ref{th1} and Theorem \ref{th2}, concerning perturbations of compact support, and of exponential decay.\\
Let $T = T^*$ be a compact operator in a Hilbert space, such that ${\rm rank}\,\one_{(0,\infty)}(T) = \infty$. Denote by $\left\{\nu_k(T)\right\}_{k = 0}^\infty$ the non-increasing sequence of the positive eigenvalues of $T$,  counted with the multiplicities. \\
Recall that $m_<(x) \leq m_>(x)$ are the eigenvalues of the matrix $m(x)$, $x \in \rd$. Since the matrix $U$ (see \eqref{58}) is unitarily equivalent to $m$, $m_{\gtrless}$ are also the eigenvalues of $U$. Next, we check that Proposition \ref{pr1} implies the following
\begin{follow} \label{f1}
Under the general assumptions of the article, there exist constants $0 < c_<^\pm \leq c_>^\pm < \infty$ and $k_0 \in \Z_+$ such that
\bel{62}
c_<^\pm \nu_{k+k_0}(P_q \A^* m_< \A P_q) \leq \pm(\lambda_{k,q}^\pm - \Lambda_q) \leq c_>^\pm \nu_{k-k_0}(P_q \A^* m_> \A P_q)
\ee
for sufficiently large $k \in \N$.
\end{follow}
\begin{proof}
It is easy to see that
\bel{62a}
0 \leq P_q W H_\pm^{-1} W P_q \leq c_\pm P_q W P_q
\ee
with
$$
c_\pm: = \|H_{\pm}^{-1/2} W H_{\pm}^{-1/2}\| \leq \sup_{x \in \rd} |m(x)(I \pm m(x))^{-1}|.
$$
Note that $0 \leq c_- < \infty$ and $0 \leq c_+ < 1$. Moreover, by \eqref{57} and the mini-max principle,
    \bel{62b}
    n_+(2s; P_q \A^* m_< \A P_q) \leq n_+(s; P_q W P_q) \leq n_+(2s; P_q \A^* m_> \A P_q), \quad s>0.
    \ee
    Now, \eqref{19}, \eqref{62a}, and \eqref{62b}, imply that for any $\varepsilon \in (0,1)$ we have
    $$
    n_+(2\lambda(1+\varepsilon); P_q \A^* m_< \A P_q) + O(1) \leq
    $$
    $$
    {\mathcal N}_q^-(\lambda) \leq
    $$
    \bel{60}
     n_+(2\lambda(1-\varepsilon); (1+c_-)P_q \A^* m_> \A P_q) + O(1),
     \ee
     $$
    n_+(2\lambda(1+\varepsilon); (1-c_+)P_q \A^* m_< \A P_q) + O(1) \leq
    $$
    $$
    {\mathcal N}_q^+(\lambda) \leq
    $$
    \bel{61}
     n_+(2\lambda(1-\varepsilon); P_q \A^* m_> \A P_q) + O(1),
     \ee
     as $\lambda \downarrow 0$, and estimates \eqref{60} - \eqref{61} yield \eqref{62} with
     $$
     c_<^- = \frac{1}{2(1+\varepsilon)}, \quad c_>^- = \frac{1+c_-}{2(1-\varepsilon)}, \quad c_<^+ = \frac{1-c_+}{2(1+\varepsilon)},
     \quad c_>^+ = \frac{1}{2(1-\varepsilon)},
     $$
     and sufficiently large $k_0 \in \N$.
    \end{proof}
    Let us now complete the proof of Theorem \ref{th1}. Let $\zeta_1 \in C_0^\infty(\rd)$, $\zeta_1 \geq 0$, $\zeta_1 =1$ on ${\rm supp}\,m_>$. Set
    $\zeta_2(x) : = \left(\max_{y\in \rd} m_>(y)\right) \zeta_1(x)$, $x \in \rd$. Evidently, $m_> \leq \zeta_2$ on $\rd$, so that
    \bel{64}
    \nu_k(P_q \A^* m_> \A P_q) \leq  \nu_k(P_q \A^* \zeta_2 \A P_q), \quad k \in \Z_+.
    \ee
    Further, by Proposition \ref{pr2}, the operator $P_q \A^* \zeta_2 \A P_q$ is unitarily equivalent to the operator $P_0 \zeta_3 P_0$ where
    $$
    \zeta_3 : = 2b\left((q+1) {\rm L}_{q+1}\left(-\frac{\Delta}{2b}\right)  +
       q {\rm L}_{q-1}\left(-\frac{\Delta}{2b}\right)\right) \zeta_2.
       $$
       Therefore,
       \bel{65}
       \nu_k(P_q \A^* \zeta_2 \A P_q) = \nu_k(P_0 \zeta_3 P_0), \quad k \in \Z_+.
       \ee
       Let $R_> > 0$ be so large that the disk $B_{R_>}(0)$ of radius $R_>$, centered at the origin contains the support of $\zeta_3$. Then,
       \bel{66}
       \nu_k(P_0 \zeta_3 P_0) \leq \max_{x \in \rd} |\zeta_3(x)| \, \nu_k(P_0 \one_{B_{R_>}(0)} P_0), \quad k \in \Z_+.
       \ee
       Putting together \eqref{64}, \eqref{65}, and \eqref{66}, we find that there exists a constant $K_> < \infty$ such that
       \bel{67}
       \nu_k(P_q \A^* m_> \A P_q) \leq K_> \nu_k(P_0 \one_{B_{R_>}(0)} P_0), \quad k \in \Z_+.
       \ee
       On the other hand,
       \bel{68}
       \nu_k(P_q \A^* m_< \A P_q) \geq \nu_k(P_q a\, m_< \,a^* P_q).
       \ee
       Applying \eqref{37}, we easily find that the operators $P_q a\, m_< a^*\, P_q$ and $2b(q+1)P_{q+1}\, m_< \, P_{q+1}$ are unitarily equivalent. Hence,
       \bel{69}
       \nu_k(P_q a m_< a^* P_q) = 2b(q+1) \nu_k(P_{q+1} m_<  P_{q+1}), \quad k \in \Z_+.
       \ee
       Further, since $m_<$ is non-negative, continuous,  and does not vanish identically, there exist $c_0 > 0$, $R_< \in (0,\infty)$,
       and $x_0 \in \rd$, such that $m_<(x) \geq c_0 \one_{{B_{R_<}}(x_0)}(x)$, $x\in \rd$. Therefore,
       \bel{70}
       \nu_k(P_{q+1} m_<  P_{q+1}) \geq  c_0 \nu_k(P_{q+1} \one_{{B_{R_<}(x_0)}}  P_{q+1}), \quad k \in \Z_+.
       \ee
       The operators  $P_{q+1} \one_{{B_{R_<}(x_0)}}  P_{q+1}$ and $P_{q+1} \one_{{B_{R_<}(0)}}  P_{q+1}$ are unitarily equivalent under the magnetic translation which maps $x_0$ into $0$ (see e.g. \cite[Eq. (4.21)]{rw}). Therefore,
       \bel{71}
       \nu_k(P_{q+1} \one_{{B_{R_<}(x_0)}}  P_{q+1}) = \nu_k(P_{q+1} \one_{{B_{R_<}(0)}}  P_{q+1}), \quad k \in \Z_+.
       \ee
       Combining \eqref{68} - \eqref{71}, we find that there exists a constant $K_<$ such that
       \bel{72}
       K_< \, \nu_k(P_{q+1} \one_{{B_{R_<}(0)}}  P_{q+1}) \leq \nu_k(P_q \A^* m_< \A P_q), \quad k \in \Z_+.
       \ee
       By \eqref{67} and \eqref{72}, it remains to study the asymptotic behaviour as $k \to \infty$ of $\nu_k(P_m  \one_{{B_{R}(0)}} P_m)$, $m \in \Z_+$ and $R \in (0,\infty)$ being fixed. This asymptotic analysis relies on the representation \eqref{53b}, and results  sufficient for our purposes, are available in the literature.
       Namely, we have
       \begin{lemma} \label{l1}
       {\rm \cite[Section 4, Corollary 2]{chkr}}
        Let $m \in \Z_+$, $R \in (0,\infty)$, $b \in (0,\infty)$. Set $\varrho : = bR^2/2$. Then
        \bel{73}
        \nu_k(P_m  \one_{{B_{R}(0)}} P_m) = \frac{e^{-\varrho} \varrho^{-m+1} k^{2m-1} \varrho^k}{m! \, k!}(1 + o(1)), \quad k \to \infty.
        \ee
        \end{lemma}
        Now, asymptotic relation \eqref{1} follows from \eqref{62}, \eqref{67}, \eqref{72}, \eqref{73}, and the elementary fact that
        $\ln{k!} = k\ln{k} + O(k)$ as $k \to \infty$.\\

        In the remaining part of this section we prove Theorem \ref{th2} concerning perturbations $m$ of exponential decay. Assume that $m$ satisfies \eqref{3}.
        Then there exist $\delta_{\gtrless} \in \re$, $\delta_< \leq \delta_>$, and $r>1$, such that
        \bel{73a}
        |x|^{\delta_<} e^{-\gamma|x|^{2\beta}} \one_{\rd\setminus B_r(0)}(x) \leq m_<(x) \leq
        \ee
$$
         m_>(x) \leq |x|^{\delta_>} e^{-\gamma|x|^{2\beta}} \one_{\rd\setminus B_r(0)}(x) + \max_{y \in \rd} m_>(y) \one_{B_r(0)}(x), \quad x \in  \rd.
$$
        Let $\eta_{\gtrless,0} \in C^{\infty}(\rd; [0,1])$ be two radially symmetric functions such that $\eta_{<,0} = 1$ on $\rd \setminus B_{r+1}(0)$,
        $\eta_{<,0} = 0$ on $B_r(0)$, and $\eta_{>,0} = 1$ on $\rd \setminus B_r(0)$,
        $\eta_{>,0} = 0$ on $B_{r-1}(0)$. For $x \in \rd$ set
        $$
        \eta_{<,1}(x) : = |x|^{\delta_<} e^{-\gamma|x|^{2\beta}} \eta_{<,0}(x),
        $$
        $$
        \eta_{>,1}(x) : = |x|^{\delta_>} e^{-\gamma|x|^{2\beta}} \eta_{>,0}(x)  + \max_{y \in \rd} m_>(y) (1-\eta_{<,0}(x)).  $$
        Evidently, $\eta_{\gtrless,1} \in C_{\rm b}^{\infty}(\rd)$, and by \eqref{73a},
        $$
         \eta_{<,1}(x) \leq m_<(x), \quad m_>(x) \leq \eta_{>,1}(x), \quad x \in \rd.
         $$
         Therefore, for $k \in \Z_+$, we have
\begin{equation}\label{o14}
\begin{gathered}
         \nu_k(P_q \A^* m_< \A P_q) \geq \nu_k(P_q \A^* \eta_{<,1} \A P_q),\\
         \nu_k(P_q \A^* m_> \A P_q) \leq \nu_k(P_q \A^* \eta_{>,1} \A P_q).
\end{gathered}
\end{equation}
         Further, set
         $$
         \eta_{\gtrless,2}: = 2b\left((q+1) {\rm L}_{q+1}\left(-\frac{\Delta}{2b}\right) + q {\rm L}_{q-1}\left(-\frac{\Delta}{2b}\right)\right)\eta_{\gtrless,1}.
         $$
         According to Proposition \ref{pr2}, the operators $P_q \A^* \eta_{\gtrless,1} \A P_q$, $q \in \Z_+$, and
         $P_0 \eta_{\gtrless,2} P_0$ are unitarily equivalent. Therefore,
         \bel{o16}
         \nu_k(P_q \A^* \eta_{\gtrless,1} \A P_q) = \nu_k(P_0 \eta_{\gtrless,2} P_0), \quad k \in \Z_+.
         \ee
         Next, a tedious but straightforward calculation shows that
         \bel{o15}
         \eta_{\gtrless,2}(x) = \eta_{\gtrless,3}(x)(1+ o(1)),  \quad |x| \to \infty,
         \ee
         where
         $$
        \eta_{\gtrless,3}(x) : = C_{q,\beta}|x|^{\delta_\gtrless} e^{-\gamma|x|^{2\beta}}
\begin{cases}
        1 \quad {\rm if} \ \beta \in (0,1/2],\\
         |x|^{2(q+1)(2\beta-1)}  \quad {\rm if} \ \beta \in (1/2, \infty),
\end{cases}
         x \in \rd\setminus\{0\},
        $$
        and $C_{q,\beta}>0$ are some constants. Even though the exact values of $C_{q,\beta}$ will not play any role in the sequel, we indicate here these values for the sake of the completeness of the exposition:
        $$
        C_{q,\beta} =
\begin{cases}
        2\Lambda_q \quad {\rm if} \quad \beta \in (0,1/2),\\
        2b\left((q+1) {\rm L}_{q+1}\left(-\frac{(2\beta\gamma)^2}{2b}\right) + q {\rm L}_{q-1}\left(-\frac{(2\beta\gamma)^2}{2b}\right)\right)
        \quad {\rm if} \quad \beta = 1/2,\vspace{3pt}\\
        \frac{(2\beta\gamma)^{2(q+1)}}{(2b)^q q!} \quad {\rm if} \quad \beta \in (1/2,\infty).
\end{cases}
    $$
    Hence, by \eqref{o15}, there exists $R \in (0,\infty)$ such that for $x \in \rd$ we have
    \bel{o43}
    \eta_{<,2} \geq \frac{1}{2}  \eta_{<,3} \one_{\rd\setminus B_R(0)} - c_< \one_{B_R(0)} = : \eta_{<,4}(x),
    \ee
    \bel{o44}
    \eta_{>,2} \leq \frac{3}{2}  \eta_{>,3} \one_{\rd\setminus B_R(0)} + c_> \one_{B_R(0)} = : \eta_{>,4}(x),
    \ee
    with $c_\gtrless : = \max_{y \in \rd} |\eta_{\gtrless,2}(y)|$. Thus, for any admissible $k \in \Z_+$ we have
    \bel{o17}
    \nu_k(P_0 \eta_{<,2} P_0) \geq \nu_k(P_0 \eta_{<,4} P_0), \quad \nu_k(P_0 \eta_{>,2} P_0) \leq \nu_k(P_0 \eta_{>,4} P_0).
    \ee
    In order to complete the proof of Theorem \ref{th2}, we need a couple of auxiliary results. For $\beta > 0$, $\mu > 0$, and $\varrho > 0$, set
    \bel{02}
    {\mathcal J}_{\beta, \mu}(k) : = \int_0^{\infty} e^{-\mu t^{\beta} -t} t^k dt, \quad  {\mathcal E}_{\varrho}(k) : = \int_0^{\varrho} e^{-t} t^k dt,
    \quad k > -1,
    \ee
    and for $\delta \in \re$, $c_0>0$ and $c_1 \in \re$, put
\begin{multline*}
    {\mathcal L}(k) = {\mathcal L}_{\beta,\mu,\varrho,\delta}(k;c_0,c_1) : =
    \frac{c_0 {\mathcal J}_{\beta, \mu}(k+\delta) + c_1 {\mathcal E}_{\varrho}(k-\delta_-)}{\Gamma(k+1)},\\ k > \max\{-1,-\delta -1\},
\end{multline*}
    where $\delta_- : = \max\{0,-\delta\}$.
    \begin{lemma} \label{l2}
    Let $\beta > 0$, $\mu > 0$,  $\varrho > 0$, $c_0>0$, and $\delta \in \re$, $c_1 \in \re$. \\
    {\rm(i)} The asymptotic relations
    \bel{o5}
    \ln{{\mathcal L}(k)} =
\begin{cases}
    -\sum_{1 \leq j < \frac{1}{1-\beta}} f_j k^{(\beta-1)j+1} + O(\ln{k}) \quad {\rm if} \quad \beta \in (0,1), \\
    -\left(\ln{(1+\mu)}\right)\,k + O(\ln{k}) \quad {\rm if} \quad \beta = 1, \vspace{2pt}\\
    -\frac{\beta-1}{\beta} k\ln{k} +  k\left(\frac{\beta-1-\ln{(\mu \beta)}}{\beta}\right)\vspace{3pt}\\\mkern30mu-\sum_{1 \leq j < \frac{\beta}{\beta-1}} g_j k^{(\frac{1}{\beta}-1)j+1} + O(\ln{k}) \quad {\rm if} \quad \beta \in (1,\infty),
\end{cases}
     \ee
    hold true as $k\to \infty$, the coefficients $f_j$ and $g_j$ being introduced in the statement of Theorem \ref{th2}.\\
    {\rm(ii)}  We have ${\mathcal L}'(k) < 0$ for sufficiently large $k$.
    \end{lemma}
    \begin{proof}
    Let at first  $\delta = 0$. Assume $\beta \in (0,1)$, $k>0$, and
    change the variable $t \mapsto ks$ in the first integral in \eqref{02}. Thus we find that
    \bel{o6}
    {\mathcal J}_{\beta, \mu}(k) = k^{k+1} \int_0^{\infty} e^{-kF(s;k^{\beta-1})} ds.
    \ee
    The function $F(s;k^{\beta-1})$ defined in \eqref{o1}, attains its unique minimum at $s_<(k^{\beta-1})$, and we have $\frac{\partial^2 F}{\partial s^2}(s_<(k^{\beta-1});k^{\beta-1}) = 1 + o(1)$, $k \to \infty$. Therefore, applying a standard argument close to the usual Laplace method for asymptotic evaluation of integrals depending on a large parameter, we easily find that
    \bel{o7}
    \int_0^{\infty} e^{-kF(s;k^{\beta-1})} ds = (2\pi)^{1/2} e^{-kF(s_<(k^{\beta-1});k^{\beta-1})} k^{-1/2}(1+o(1)), \quad k \to \infty.
    \ee
    Bearing in mind that $F(s_<(k^{\beta-1});k^{\beta-1}) = f(k^{\beta-1})$ (see \eqref{o4}), $f(0) = 1$, and
    \bel{o3}
    \ln{\Gamma(k+1)} = k\ln{k} -k + \frac{1}{2} \ln{k} + O(1), \quad k \to \infty,
    \ee
   (see e.g. \cite[Eq. 6.1.40]{abst}), we find that \eqref{o6} -- \eqref{o7} imply
    \begin{align}
    \ln{\left(\frac{{\mathcal J}_{\beta, \mu}(k)}{\Gamma(k+1)}\right)} & = k - kf(k^{\beta -1}) + O(\ln{k})\nonumber\\ &=
    k-k\sum_{0 \leq j < \frac{1}{1-\beta}} \frac{1}{j!} \frac{d^j f}{d\epsilon^j}(0) k^{(\beta-1)j}+ O(\ln{k})\nonumber\\ &=
    -\sum_{1 \leq j < \frac{1}{1-\beta}} \frac{1}{j!} \frac{d^jf}{d\epsilon^j}(0) k^{(\beta-1)j+1}+ O(\ln{k})\nonumber\\ &=
    -\sum_{1 \leq j < \frac{1}{1-\beta}} f_j k^{(\beta-1)j + 1}+ O(\ln{k}), \quad k \to \infty.
    \label{o7a}
    \end{align}
    In the case $\beta = 1$, we simply have
    $$
    \frac{{\mathcal J}_{\beta, \mu}(k)}{\Gamma(k+1)} = \frac{1}{\Gamma(k+1)}\int_0^\infty e^{-(\mu+1)t} t^k dt = (\mu + 1)^{-k-1},
    $$
    i.e.
    \bel{o7b}
    \ln{\left( \frac{{\mathcal J}_{\beta, \mu}(k)}{\Gamma(k+1)}\right)} = -(\ln{(1+\mu)}) k + O(1), \quad k \to \infty.
    \ee
    Let now $\beta \in (1,\infty)$. Changing the variable $t \mapsto k^{1/\beta}s$ with $k>0$ in  \eqref{02},
      we find
    \bel{o11}
    {\mathcal J}_{\beta, \mu}(k) : = k^{(k+1)/\beta} \int_0^{\infty} e^{-kG(s;k^{(\frac{1}{\beta}-1)})} ds.
    \ee
    The function $G(s;k^{\frac{1}{\beta}-1})$ defined in \eqref{o9}, attains its unique minimum at $s_>(k^{\frac{1}{\beta}-1})$, and we have
    $$
    \frac{\partial^2 G}{\partial s^2}(s_>(k^{\frac{1}{\beta}-1}),k^{\frac{1}{\beta}-1}) = \beta (\mu \beta)^{2/\beta}(1 + o(1)), \quad k \to \infty.
     $$
     Arguing as in the derivation of \eqref{o7}, we obtain
    \bel{o12}
    \int_0^{\infty} e^{-kG(s;k^{\frac{1}{\beta}-1})} ds = \sqrt{2\pi \beta} \, (\mu \beta)^{-1/\beta}  e^{-kG(s_>(k^{\frac{1}{\beta}-1});k^{\frac{1}{\beta}-1})} k^{-1/2}(1+o(1)), \quad k \to \infty.
    \ee
    Bearing in mind that $G(s_>(k^{\frac{1}{\beta}-1});k^{\frac{1}{\beta}-1}) = g(k^{\frac{1}{\beta}-1})$ (see \eqref{o10}), and $g(0) = \frac{1+\ln{(\mu \beta)}}{\beta}$, we find that  \eqref{o11}, \eqref{o12}, and \eqref{o3}, imply
    \begin{align}
    \ln{\left(\frac{{\mathcal J}_{\beta, \mu}(k)}{\Gamma(k+1)}\right)} & = -\frac{\beta-1}{\beta} k\ln{k} + k - kg(k^{\frac{1}{\beta}-1}) + O(\ln{k})\nonumber\\ &\mkern-70mu=
   -\frac{\beta-1}{\beta} k\ln{k} +  k-k\sum_{0 \leq j < \frac{\beta}{\beta-1}} \frac{1}{j!} \frac{d^j g}{d\epsilon^j}(0) k^{(\frac{1}{\beta}-1)j}+ O(\ln{k})\nonumber\\ &\mkern-70mu=
   -\frac{\beta-1}{\beta} k\ln{k} +  k(1-g(0))-\sum_{1 \leq j < \frac{\beta}{\beta-1}} \frac{1}{j!} \frac{d^j g}{d\epsilon^j}(0) k^{(\frac{1}{\beta}-1)j+1}+ O(\ln{k})\nonumber\\ &\mkern-70mu=
   -\frac{\beta-1}{\beta} k\ln{k} +  k\left(\frac{\beta-1-\ln{(\mu \beta)}}{\beta}\right)-\sum_{1 \leq j < \frac{\beta}{\beta-1}} g_j k^{(\frac{1}{\beta}-1)j+1}+ O(\ln{k}),
    \label{o7c}
    \end{align}
    as $k \to \infty$.
    Let us now consider general $\delta \in \re$. By \eqref{o3},
    \bel{o7f}
    \ln{\left(\frac{\Gamma(k+\delta+1)}{\Gamma(k+1)}\right)} = \delta \ln{k} + O(1), \quad k \to \infty.
    \ee
    Putting together \eqref{o7a}, \eqref{o7b}, \eqref{o7c}, and  \eqref{o7f}, we find that
    \bel{o7g}
     \ln{\left(\frac{{\mathcal J}_{\beta, \mu}(k+\delta)}{\Gamma(k+1)}\right)} -  \ln{\left(\frac{{\mathcal J}_{\beta, \mu}(k)}{\Gamma(k+1)}\right)} = O(\ln{k}), \quad k \to \infty.
     \ee
     Finally, by \eqref{73}, we easily find that for each $\delta \in \re$ fixed, we have
     \bel{o7d}
     \frac{{\mathcal E}_\varrho(k-\delta_-)}{\Gamma(k+1)} = o\left(\frac{{\mathcal J}_{\beta, \mu}(k + \delta)}{\Gamma(k+1)}\right), \quad k \to \infty.
     \ee
     The combination of \eqref{o7a}, \eqref{o7b}, \eqref{o7c}, \eqref{o7g}, and \eqref{o7d} implies \eqref{o5}. \\
     (ii) We have
     $$
     {\mathcal L}'(k) =
     $$
     $$
      c_0\left( \frac{{\mathcal J}'_{\beta, \mu}(k+\delta)}{\Gamma(k+1)} - \frac{\Gamma'(k+1)}{\Gamma(k+1)^2} {\mathcal J}_{\beta, \mu}(k+\delta)\right) +
      $$
      \bel{o40}
      c_1\left(\frac{{\mathcal E}'_{\varrho}(k-\delta_-)}{\Gamma(k+1)} - \frac{\Gamma'(k+1)}{\Gamma(k+1)^2} {\mathcal E}_{\varrho}(k-\delta_-)\right),
     \ee
     $$
     {\mathcal J}'_{\beta, \mu}(k) = \int_0^\infty e^{-\mu t^\beta - t} t^k \ln{t}\, dt,
     \quad  {\mathcal E}'_{\varrho}(k) = \int_0^\varrho e^{-t} t^k \ln{t} \, dt,
     $$
     and
     $$
     \frac{\Gamma'(k+1)}{\Gamma(k+1)} = \ln{k} + \frac{1}{2k} + O(k^{-2}), \quad k \to \infty,
     $$
     (see e.g. \cite[Eq. 6.3.18]{abst}). Performing an asymptotic analysis similar to the one in the proof of the first part of the lemma, we find that there exists a function $\Psi = \Psi_{\beta, \mu , \delta}$ such that $\Psi(k) < 0$ for $k$ large enough, and
     \bel{o41}
     \frac{{\mathcal J}'_{\beta, \mu}(k+\delta)}{\Gamma(k+1)} - \frac{\Gamma'(k+1)}{\Gamma(k+1)^2} {\mathcal J}_{\beta, \mu}(k+\delta) = \Psi(k)(1 + o(1)),
     \ee
      \bel{o42}
       \frac{{\mathcal E}'_{\varrho}(k-\delta_-)}{\Gamma(k+1)} - \frac{\Gamma'(k+1)}{\Gamma(k+1)^2} {\mathcal E}_{\varrho}(k-\delta_-) = o(\Psi(k)),
       \ee
       as $k \to \infty$. Putting together \eqref{o40}, \eqref{o41}, and \eqref{o42}, we conclude that ${\mathcal L}'(k) < 0$ for sufficiently large $k$.
     \end{proof}
    Taking into account the definition of the functions $\eta_{\gtrless, 4}$ in \eqref{o43} - \eqref{o44}, the mini-max principle, representation \eqref{53a}, as well as Lemma \ref{l2} (ii), we find that there exist constants $c_{j,\gtrless} > 0$, $j=0,1$, $\tilde{\delta}_\gtrless \in \re$, and $k_0 \in \Z_+$, such that
\begin{equation}\label{o18}
\begin{gathered}
    \nu_k(P_0 \eta_{<,4} P_0) \geq {\mathcal L}_{\beta,\mu,\varrho,\tilde{\delta}_<}(k + k_0; c_{0,<}, -c_{1,<}), \\
    \nu_k(P_0 \eta_{>,4} P_0) \leq {\mathcal L}_{\beta,\mu,\varrho,\tilde{\delta}_>}(k; c_{0,>}, c_{1,>}),
\end{gathered}
\ee
    for $\mu = \gamma(2/b)^{\beta}$, $\varrho = bR^2/2$, and sufficiently large $k \in \Z_+$.\\
     Putting together \eqref{62}, \eqref{o14}, \eqref{o16}, \eqref{o17}, \eqref{o18}, and \eqref{o5}, we obtain \eqref{4} -- \eqref{6}.

%%%%%%%%%%%%%%%%%%%%%%%%%%%%%%%%%%%%%%%%%%%%%%%%%%%%%%%%%%%%%%%%%%%%%%%%%%%%%%%
\section{Proof of Theorem \ref{th3}}
\label{ss35} \setcounter{equation}{0}
%%%%%%%%%%%%%%%%%%%%%%%%%%%%%%%%%%%%%%%%%%%%%%%%%%%%%%%%%%%%%%%%%%%%%%%%%%%%%%%
Estimates \eqref{19} combined with the Weyl inequalities \eqref{wi} and the mini-max principle, entail
$$
n_+(\lambda(1+\varepsilon); P_q W P_q) + O(1) \leq
$$
$$
{\mathcal N}_q^-(\lambda) \leq
$$
\bel{o20a}
n_+(\lambda(1-\varepsilon)^2; P_q W P_q) + n_+(\lambda \varepsilon (1 - \varepsilon); P_q W H_-^{-1} W P_q) + O(1),
\ee
$$
n_+(\lambda(1+\varepsilon)^2; P_q W P_q) - n_+(\lambda \varepsilon (1 + \varepsilon); P_q W H_+^{-1} W P_q) + O(1) \leq
$$
$$
{\mathcal N}_q^+(\lambda) \leq
$$
\bel{o20b}
n_+(\lambda(1-\varepsilon); P_q W P_q)  + O(1),
\ee
as $\lambda \downarrow 0$. It is easy to check that we have
$$
P_q W H_{\pm}^{-1} W P_q \leq C_{1,\pm} P_q \A^* \langle \cdot \rangle^{-2\rho} \A P_q
$$
with
$$
C_{1,\pm} : =  \|H_0^{1/2} H_\pm^{-1/2}\|^2 \left(\sup_{x \in \rd} \langle x \rangle^{\rho}m_>(x)\right)^2.
$$
Therefore, for any $s>0$,
    \bel{o21}
n_+(s;P_q W H_{\pm}^{-1} W P_q) \leq n_+(s; C_{1,\pm} P_q \A^* \langle \cdot \rangle^{-2\rho} \A P_q).
    \ee
    Further, by Proposition \ref{pr2}, the operator $P_q W P_q$ (resp., $P_q \A^* \langle \cdot \rangle^{-2\rho} \A P_q$) is unitarily equivalent to $\frac{1}{2} P_0 w_q\left(U\right) P_0$
(resp., to $P_0 w_q\left(\langle \cdot \rangle^{-2\rho} I\right) P_0$). Hence, for any $s>0$,
    \bel{o22}
    n_+(s; P_q W P_q) = n_+(2s; P_0 w_q\left(U\right) P_0),
    \ee
    \bel{o23}
    n_+(s; P_q \A^* \langle \cdot \rangle^{-2\rho} \A P_q) = n_+(s; P_0 w_q\left(\langle \cdot \rangle^{-2\rho} I\right) P_0) \leq n_+(s; C_2 P_0 \langle \cdot \rangle^{-2\rho} P_0)
    \ee
    with $C_2 : = \sup_{x \in \rd}\langle x \rangle^{2\rho} |w_q\left(\langle x \rangle^{-2\rho} I\right)|$. Now, write
    $$
    \frac{1}{2} w_q\left(U\right) = {\mathcal T}_q + \tilde{\mathcal T}_q,
    $$
    the symbol ${\mathcal T}_q$ being defined in \eqref{o24}, and note the crucial circumstance that $\tilde{\mathcal T}_q \in {\mathcal S}^{-\rho-2}(\rd)$. Then the Weyl inequalities \eqref{wi} entail $$
    n_+(s(1+\varepsilon); P_0 {\mathcal T}_q P_0) - n_-(s\varepsilon; P_0 \tilde{\mathcal T}_q P_0) \leq
    $$
    $$
    n_+(2s; P_0 w_q\left(U\right) P_0) \leq
    $$
    \bel{o30}
    n_+(s(1-\varepsilon); P_0 {\mathcal T}_q P_0) + n_+(s\varepsilon; P_0 \tilde{\mathcal T}_q P_0),
    \ee
    for any $s>0$ and $\varepsilon \in (0,1)$. Evidently,
    \bel{o25}
    n_\pm(s; P_0 \tilde{\mathcal T}_q P_0) \leq n_+(s; C_3 P_0 \langle \cdot \rangle^{-\rho - 2} P_0), \quad s>0,
    \ee
    with $C_3 : = \sup_{x \in \rd} \langle x \rangle^{\rho + 2} |\tilde{\mathcal T}_q(x)|$. Recalling Proposition \ref{pr4}, we find that we have reduced the asymptotic analysis of ${\mathcal N}_q^{\pm}(\lambda)$ as $\lambda \downarrow 0$ to the eigenvalue asymptotics for a $\Psi$DO with elliptic anti-Wick symbol of negative order. The spectral asymptotics for operators of this type has been extensively studied in the literature since the 1970s. In particular, we have the following
    \begin{pr} \label{pr5}
    Let $0 \leq \psi \in {\mathcal S}^{-\rho}(\rd)$, $\rho>0$. Assume that there exists $0 < \psi_0 \in C^\infty({\mathbb S}^1)$ such that
    $ \lim_{|x| \to \infty} |x|^{\rho} \psi(x) = \psi_0(x/|x|)$.
    Then we have
    \bel{o31}
    n_+(\lambda; {\rm Op}^{\rm aw}(\psi)) = (2\pi)^{-1} \Phi_\psi(\lambda)(1+o(1)), \quad  \lambda \downarrow 0,
    \ee
    which is equivalent to
    $$
    \lim_{\lambda \downarrow 0} \lambda^{2/\rho}  n_+(\lambda; {\rm Op}^{\rm aw}(\psi)) = {\mathcal C}(\psi_0) : = \frac{1}{4\pi} \int_0^{2\pi} \psi_0(\cos{\theta}, \sin{\theta})^{2/\rho} d\theta.
    $$
    \end{pr}
    \begin{proof}
    Evidently, for each $\varepsilon \in (0,1)$ there exist real functions $\psi_{\pm, \varepsilon} \in C^\infty(\rd)$ such that
    $$
    \psi_{-,\varepsilon}(x) \leq \psi(x) \leq \psi_{+,\varepsilon}(x), \quad x \in \rd,
    $$
    $$
    \psi_{\pm, \varepsilon}(x) = (1 \mp \varepsilon)^{-1} |x|^{-\rho} \psi_0(x/|x|), \quad x \in \rd, \quad |x| \geq R,
    $$
    for some $R\in (0,\infty)$. Applying the monotonicity of the anti-Wick quantization with respect to the symbol (see e.g. \cite[Proposition 24.1]{shu}), the mini-max principle, and the Weyl inequalities, we obtain
    $$
    n_+((1+\varepsilon)\lambda; {\rm Op}^{\rm w}(\psi_{-,\varepsilon})) - n_-(\varepsilon\lambda; ({\rm Op}^{\rm aw}(\psi_{-,\varepsilon})-{\rm Op}^{\rm w}(\psi_{-,\varepsilon}))) \leq
    $$
    $$
    n_+(\lambda; {\rm Op}^{\rm aw}(\psi)) \leq
    $$
    \bel{j5}
    n_+((1-\varepsilon)\lambda; {\rm Op}^{\rm w}(\psi_{+,\varepsilon})) + n_+(\varepsilon \lambda; ({\rm Op}^{\rm aw}(\psi_{+,\varepsilon})-{\rm Op}^{\rm w}(\psi_{+,\varepsilon}))).
    \ee
    By \cite{daurob}, we have the following semiclassical result
     \bel{o32}
    n_+(\lambda; {\rm Op}^{\rm w}(\psi_{\pm, \varepsilon})) = (2\pi)^{-1} \Phi_{\psi_{\pm, \varepsilon}}(\lambda)(1+o(1)), \quad \lambda \downarrow 0.
    \ee
   Further, by \cite[Theorem 24.1]{shu} the differences ${\rm Op}^{\rm aw}(\psi_{\pm, \varepsilon}) - {\rm Op}^{\rm w}(\psi_{\pm, \varepsilon})$ are $\Psi$DOs of lower order than ${\rm Op}^{\rm w}(\psi_{\pm, \varepsilon})$, so that we easily obtain
    \bel{o31a}
    \lim_{\lambda \downarrow 0} \lambda^{2/\rho} n_{\pm}(\varepsilon \lambda; ({\rm Op}^{\rm aw}(\psi_{\pm,\varepsilon})-{\rm Op}^{\rm w}(\psi_{\pm,\varepsilon})) = 0, \quad \varepsilon > 0.
    \ee
    Now, \eqref{j5} -- \eqref{o31a} imply
    $$
    (1 + \varepsilon)^{-4/\rho} {\mathcal C}(\psi_0) \leq
    \liminf_{\lambda \downarrow 0} \lambda^{2/\rho} n_+(\lambda; {\rm Op}^{\rm aw}(\psi)) \leq
    $$
    $$
    \limsup_{\lambda \downarrow 0} \lambda^{2/\rho} n_+(\lambda; {\rm Op}^{\rm aw}(\psi)) \leq
    (1 - \varepsilon)^{-4/\rho} {\mathcal C}(\psi_0),
    $$
    for $\varepsilon \in (0,1)$.
    Letting $\varepsilon \downarrow 0$, we obtain \eqref{o31}.

    \end{proof}
    By Propositions \ref{pr4} and \ref{pr5}, we have
    $$
    n_+(\lambda; P_0 {\mathcal T}_q P_0) = n_+(\lambda; {\rm Op}^{\rm aw} ({\mathcal T}_{q,b})) =
    $$
    \bel{o26}
    \frac{1}{2\pi} \Phi_{{\mathcal T}_{q,b}}(\lambda)(1+o(1)) = \frac{b}{2\pi} \Phi_{{\mathcal T}_q}(\lambda)(1+o(1)), \quad \lambda \downarrow 0,
    \ee
    with  ${\mathcal T}_{q,b} = {\mathcal T}_{q} \circ {\mathcal R}_b$, ${\mathcal R}_b$ being defined in \eqref{ggg}.
    Finally, for $\rho_0 > \rho$, we have
   \bel{o28}
    n_+(\lambda; P_0 \langle \cdot \rangle^{-{\rho_0}} P_0)
    = O(\lambda^{-2/\rho_0}) = o(\Phi_{{\mathcal T}_q}(\lambda)), \quad \lambda \downarrow 0.
    \ee
     Now,   \eqref{10} easily follows from \eqref{o20a} -- \eqref{o31}, \eqref{o26}, and \eqref{o28}. The equivalence of \eqref{11} and \eqref{j1} can be checked by arguing as in the proof of \cite[Proposition 13.1]{shu}.
    \appendix
\section{Compactness of the Resolvent Differences}
\label{app} \setcounter{equation}{0}
{\em A priori}, the operators $H_0$ and $H_\pm$, self-adjoint in $L^2(\rd)$, could be defined as the Friedrichs extensions of the operators $\sum_{j=1,2} \Pi_j^2$ and
$\sum_{j,k=1,2} \Pi_j g_{jk}^{\pm} \Pi_k$ defined on $C_0^{\infty}(\rd)$. Such a definition implies immediately that
$$
{\rm Dom}\,H_0^{1/2} = {\rm Dom}\,H_\pm^{1/2} =
\left\{u \in L^2(\rd) \, | \, \Pi_j u \in L^2(\rd), \; j = 1,2\right\},
$$
and that  the operators $H_\pm^{1/2} H_0^{-1/2}$ and $H_0^{1/2} H_\pm^{-1/2}$ are bounded.
By \cite[Proposition A.2]{gms}, the operators $H_0$ and $H_\pm$ are essentially self-adjoint on $C_0^\infty(\rd)$, and have a common domain
$$
{\rm Dom}\,H_0 = {\rm Dom}\,H_\pm =
\left\{u \in L^2(\rd) \, | \, \Pi_j \Pi_k u \in L^2(\rd), \; j,k = 1,2\right\}.
$$
Let us now prove the compactness of the operator $H_0^{-1} - H_\pm^{-1}$ in $L^2(\rd)$. Since we have
$$
H_0^{-1} - H_\pm^{-1} = \pm H_0^{-1} W H_\pm^{-1} = \pm H_0^{-1} W  H_0^{-1}  H_0 H_\pm^{-1},
$$
it suffices to prove the compactness of  $H_0^{-1} W  H_0^{-1}$. The operators $H_0^{-1} W  H_0^{-1} = \frac{1}{2} H_0^{-1} \A^* U \A  H_0^{-1}$
and $\frac{1}{2} H_0^{-1} \A^* m_> \A  H_0^{-1}$ are bounded, self-adjoint, and positive. Moreover,
\bel{a1}
H_0^{-1} \A^* U \A  H_0^{-1} \leq H_0^{-1} \A^* m_> \A  H_0^{-1}.
    \ee
    On the other hand,
    \bel{a2}
    H_0^{-1} \A^* m_> \A  H_0^{-1} = H_0^{-1} a^* m_> a  H_0^{-1} + H_0^{-1} a m_> a^*  H_0^{-1}.
    \ee
    By \eqref{a1} and \eqref{a2}, it suffices to prove the compactness of the operator $m_>^{1/2} a^* H_0^{-1}$. We have
    $$
    m_>^{1/2} a^* H_0^{-1} = m_>^{1/2} H_0^{-1/2} \left(H_0^{-1/2} a^* + 2b H_0^{-1/2} a^* H_0^{-1}\right).
    $$
    The operator $H_0^{-1/2} a^* + 2b H_0^{-1/2} a^* H_0^{-1}$ is bounded, so that it suffices to prove the compactness of $m_>^{1/2} H_0^{-1/2}$ which follows from  $m_> \in L^{\infty}(\rd)$,  $\lim_{|x| \to \infty} m_>(x) = 0$, and the diamagnetic inequality (see e.g. \cite[Theorem 2.5]{ahs}).   \\

{\bf Acknowledgements}.
The final version of this work has been done during authors' visit to the Isaac Newton Institute, Cambridge, UK, in January 2015. The authors thank the Newton Institute for financial support and hospitality.
The partial
support by the Chilean Scientific Foundation {\em Fondecyt}
under Grant 1130591, by {\em N\'ucleo Milenio de F\'isica Matem\'atica} RC120002, and by the Faculty of Mathematics, PUC, Santiago de Chile, is gratefully acknowledged as well.

\bigskip
{\sc Tom\'as Lungenstrass, Georgi Raikov}\\
 %Departamento de Matem\'aticas,
 Facultad de
Matem\'aticas\\ Pontificia Universidad Cat\'olica de Chile\\
Vicu\~na Mackenna 4860\\ Santiago de Chile\\
E-mails: tlungens@mat.puc.cl, graikov@mat.puc.cl

\end{document}